\documentclass[a4paper,11pt]{amsart}
\usepackage{amssymb,amscd}
\usepackage[pdftex]{graphicx}
\usepackage{enumitem}
\usepackage[capitalize]{cleveref}
\numberwithin{equation}{section}

\newtheorem{lem}[equation]{Lemma}
\newtheorem{prop}[equation]{Proposition}
\newtheorem{thm}[equation]{Theorem}
\newtheorem{fc}[equation]{F\o{}lner criterion}
\theoremstyle{definition}

\newtheorem{exa}[equation]{Example}
\newtheorem{exas}[equation]{Examples}
\newtheorem{que}[equation]{Question}
\newtheorem{rem}[equation]{Remark}
\newtheorem{rems}[equation]{Remarks}
\def\N{\mathbb N}
\def\R{\mathbb R}
\def\ve{\varepsilon}
\def\vf{\varphi}
\def\la{\langle}
\def\ra{\rangle}
\newcommand{\ess}{\operatorname{ess}}
\newcommand{\dive}{\operatorname{div}}

\newcommand{\inte}{\operatorname{int}}
\newcommand{\grad}{\operatorname{grad}}
\newcommand{\Ric}{\operatorname{Ric}}

\begin{document}


\title{Bottom of spectra and amenability of coverings}
\author{Werner Ballmann}
\address
{WB: Max Planck Institute for Mathematics,
Vivatsgasse 7, 53111 Bonn
and Hausdorff Center for Mathematics,
Endenicher Allee 60, 53115 Bonn}
\email{hwbllmnn\@@mpim-bonn.mpg.de}
\author{Henrik Matthiesen}
\address
{HM: Max Planck Institute for Mathematics,
Vivatsgasse 7, 53111 Bonn,
and Hausdorff Center for Mathematics,
Endenicher Allee 60, 53115 Bonn}
\email{hematt\@@mpim-bonn.mpg.de}
\author{Panagiotis Polymerakis}
\address{PP: Institut f{\"u}r Mathematik,
Humboldt-Universit\"at zu Berlin,
Unter den Linden 6, 10099 Berlin.}
\email{polymerp@hu-berlin.de}


\thanks{\emph{Acknowledgments.}
We are grateful to the Max Planck Institute for Mathematics and the Hausdorff Center for Mathematics for their support and hospitality.}

\date{\today}


\begin{abstract}
For a Riemannian covering $\pi\colon M_1\to M_0$,
the bottoms of the spectra of $M_0$ and $M_1$ coincide if the covering is amenable.
The converse implication does not always hold.
Assuming completeness and a lower bound on the Ricci curvature,
we obtain a converse under a natural condition on the spectrum of $M_0$.
\end{abstract}

\maketitle

\section{Introduction}
\label{intro}

We are interested in the behaviour of the bottom of the spectrum of Laplace and Schr\"odinger operators under coverings.
To set the stage, let $M$ be a simply connected and complete Riemannian manifold
and $\pi_0\colon M\to M_0$ and $\pi_1\colon M\to M_1$ be Riemannian subcovers of $M$.
Let $\Gamma_0$ and $\Gamma_1$ be the groups of covering transformations of $\pi_0$ and $\pi_1$, respectively, and assume that $\Gamma_1\subseteq\Gamma_0$.
Then the resulting Riemannian covering $\pi\colon M_1\to M_0$ satisfies $\pi\circ\pi_1=\pi_0$.
Under these circumstances, we always have
\begin{align}\label{geq}
  \lambda_0(M_1) \ge \lambda_0(M_0),
\end{align}
see e.g.\;\cite[Theorem 1.1]{BMP} (and \cref{secpre} for notions and notations).
Recall also that any local isometry between complete and connected Riemannian manifolds is a Riemannian covering and, therefore, fits into our schema.

We say that the covering $\pi$ is \emph{amenable} if the right action of $\Gamma_0$ on $\Gamma_1\backslash\Gamma_0$ is amenable.
If $\pi$ is normal, that is, if $\Gamma_1$ is a normal subgroup of $\Gamma_0$,
then this holds if and only if $\Gamma_1\backslash\Gamma_0$ is an amenable group.
If $\pi$ is amenable, then
\begin{align}\label{ame}
  \lambda_0(M_1) = \lambda_0(M_0),
\end{align}
see \cite[Theorem 1.2]{BMP}.
The problem whether, conversely, equality implies amenability of the covering is quite sophisticated,
as Theorems \ref{gefin} and \ref{coco}, \cref{exasim}, and the examples on pages 104--105 in \cite{Br2} show.
In the case where $M_0$ is compact and $\pi$ is the universal covering (that is, $\pi=\pi_0$), amenability has been established by Brooks \cite[Theorem 1]{Br1}.
(A proof avoiding geometric measure theory is contained in \cite{Po}.)
Theorem 2 of Brooks in \cite{Br2} and Th\'eor\`eme 4.3 of Roblin and Tapie in \cite{RT} include normal Riemannian coverings of non-compact manifolds, but impose spectral conditions on $M_0$ and $\pi$, which it might be difficult to verify, and restrictions on the topology of $M_0$.
At the expense of requiring a lower bound on the Ricci curvature, we eliminate topological assumptions altogether and replace the spectral assumptions in \cite{Br2} and \cite{RT} by a weaker and natural condition on the bottom  $\lambda_{\ess}(M_0)$ of the essential spectrum of $M_0$.

\begin{thm}\label{notam}
Suppose that the Ricci curvature of $M$ is bounded from below
and that $\lambda_{\ess}(M_0)>\lambda_0(M_0)$.
Then
\begin{align*}
  \lambda_0(M_1) = \lambda_0(M_0)
\end{align*}
if and only if the covering $\pi\colon M_1\to M_0$ is amenable.
\end{thm}

\cref{notam} gives a positive answer to the speculations of Brooks on page 102 of \cite{Br2}.
Theorems \ref{gefin} and \ref{coco} and \cref{exasim} show that the assumption $\lambda_{\ess}(M_0)>\lambda_0(M_0)$ is sensible.
We do not know, however, whether the additional assumption on the Ricci curvature is necessary.

\begin{exas}\label{exafir}
1) If $M_0$ is compact, then the Ricci curvature of $M_0$ is bounded and $\lambda_{\ess}(M_0)=\infty>0=\lambda_0(M_0)$.

2) If $M_0$ is non-compact, of finite volume, and with sectional curvature $-b^2\le K_M\le -a^2$,
where $b>a>0$,
then $\lambda_0(M_0)=0$ and $\Ric_M\ge(1-m)b^2$, where $m$ denotes the dimension of $M$.
Moreover,
\begin{align}\label{exafire}
  \lambda_{\ess}(M_0) \ge a^2(m-1)^2/4,
\end{align}
and hence $\lambda_{\ess}(M_0)>\lambda_0(M_0)$.
For the convenience of the reader, we will present a short proof of \eqref{exafire} at the end of the article.
\end{exas}

A hyperbolic manifold $M$ of dimension $m$ is called \emph{geometrically finite} if the action of its covering group $\Gamma$ on the hyperbolic space $H^m$ admits a fundamental domain $F\subseteq H^m$ which is bounded by finitely many totally geodesic hyperplanes.
By the work of Lax and Phillips (\cite[p.\;281]{LP}), $\lambda_{\ess}(M)=(m-1)^2/4$ if $M$ is geometrically finite of infinite volume.

\begin{thm}\label{gefin}
Let $\pi\colon M_1\to M_0$ be a Riemannian covering of hyperbolic manifolds of dimension $m$ with corresponding covering groups $\Gamma_1\subseteq\Gamma_0$ of isometries of $H^m$. 
Assume that $M_0$ is geometrically finite of infinite volume.
Then we have:
\begin{enumerate}
\item\label{gefia}
If $\lambda_0(M_0)<(m-1)^2/4$, then $\lambda_0(M_1) = \lambda_0(M_0)$
if and only if $\pi$ is amenable.
\item\label{gefib}
If $\lambda_0(M_0)=(m-1)^2/4$, then $\lambda_0(M_1) = \lambda_0(M_0)$.
\end{enumerate}
\end{thm}

The first assertion of \cref{gefin} follows immediately from \cref{notam} and the identification $\lambda_{\ess}(M_0)=(m-1)^2/4$ by Lax and Phillips quoted above,
the second is an incarnation of the general observation stated in \cref{equess}.\ref{equb} below, using that $\lambda_0(H^m)=(m-1)^2/4$.

\begin{rems}
1)
We say that a geometrically finite hyperbolic manifold $M=\Gamma\backslash H^m$ is \emph{convex cocompact} if it does not have cusps or, equivalently, if $\Gamma$ does not contain parabolic isometries.
\cref{gefin}.\ref{gefia} is due to Brooks in the convex cocompact case.
See (\cite[Theorem 3]{Br2}) and also \cite[Th\'eor\`eme 0.2]{RT}.

2)
The \emph{critical exponent} $\delta(\Gamma)$ of a discrete group $\Gamma$ of isometries of $H^m$ is the infimum of the set of $s\in\R$ such that the Poincar\'e series
\begin{align*}
  g(x,y,s) = \sum_{\gamma\in\Gamma} e^{-sd(x,\gamma(y))}
\end{align*}
converges for all $x,y\in H^m$.
Using Sullivan's \cite[Theorem 2.17]{Su},
the assumptions on $\lambda_0(M_0)$ in \cref{gefin} may be reformulated in terms of the critical exponent of $\Gamma_0$.
Namely
\begin{align*}
  \lambda_0(M_0) = \delta(\Gamma_0)(m-1-\delta(\Gamma_0)) < (m-1)^2/4 = \lambda_{\ess}(M_0)
\end{align*}
if $\delta(\Gamma_0)>(m-1)/2$ and $\lambda_0(M_0) = (m-1)^2/4$ if $\delta(\Gamma_0)\le(m-1)/2$.
\end{rems}

Let $M$ be the interior of a compact and connected manifold $N$ with non-empty boundary and $h$ be a Riemannian metric on $N$.
Let $\rho\ge0$ be a smooth non-negative function on $N$ defining $\partial N$, that is,
\begin{align}
  \partial N = \{\rho=0\} \hspace{2mm}\text{and}\hspace{2mm} \partial_\nu\rho > 0
\end{align}
along $\partial N$,
where $\nu$ denotes the inner normal of $N$ along $\partial N$ with respect to $h$.
Consider the conformally equivalent metric 
\begin{align}\label{mazzeo}
   g = \rho^{-2}h
\end{align}
on $M$.
The metric $g$ is complete since the factor $\rho^{-2}$ causes $\partial N$ to have infinite distance to any point in $M$.
Metrics of this kind were introduced by Mazzeo, who named them \emph{conformally compact}.
In \cite[Theorem 1.3]{Ma1}, he obtains that the essential spectrum of $g$ is $[a^2(m-1)^2/4,\infty)$,
where $a=\min\partial_\nu\rho>0$ and $m=\dim M$.
In particular, $\lambda_{\ess}(g)=a^2(m-1)^2/4$.

\begin{thm}\label{coco}
Let $\pi\colon M_1\to M_0$ be a Riemannian covering of manifolds of dimension $m$ with corresponding covering groups $\Gamma_1\subseteq\Gamma_0$ of isometries of their universal covering space $M$.
Assume that $M_0$ is conformally compact with $a=\min\partial_\nu\rho$ as above.
Then we have:
\begin{enumerate}
\item\label{cocoa}
If $\lambda_0(M_0)<a^2(m-1)^2/4$, then $\lambda_0(M_1) = \lambda_0(M_0)$
if and only if $\pi$ is amenable.
\item\label{cocob}
If $\lambda_0(M_0)=a^2(m-1)^2/4$, then $\lambda_0(M_1) = \lambda_0(M_0)$.
\end{enumerate}
\end{thm}

The first assertion of \cref{coco} follows immediately from \cref{notam} together with Mazzeo's $\lambda_{\ess}(M_0)=a^2(m-1)^2/4$ quoted above,
where we note that the sectional curvature of $M_0$ is bounded from above and below.
The second assertion of \cref{coco} is proved in \cref{secfin}.

\begin{rem}
By changing the metric on a compact part of $M_0$ appropriately, it is easy to obtain examples which satisfy the first assertion of \cref{coco}.
The same remark applies to \cref{gefin}.
\end{rem}

\begin{exa}[concerning \cref{notam}]\label{exasim}
Let $P$ be a compact and connected manifold of dimension $m$ with connected boundary $\partial P=:N_0$.
Assume that the fundamental group of $N_0$ is amenable; e.g., $N_0=S^{m-1}$. 
Let $U\cong[0,\infty)\times N_0$ be a collared neighborhood of $N_0\cong\{0\}\times N_0$ in $P$.
Let $g_0$ be a Riemannian metric on $M_0=P\setminus N_0$, which is equal to $dx^2+h_0$ along $V_0=U\setminus N_0\cong(0,\infty)\times N_0$,
where we write elements of $V_0$ as pairs $(x,y)$ with $x\in(0,\infty)$ and $y\in N_0$
and where $h_0$ is a Riemannian metric on $N_0$.
Since $N_0$ is compact, we have $\lambda_0(V_0)=0$.
Since $\lambda_0(M_0)\le\lambda_0(V_0)$,
we conclude that $\lambda_0(M_0)=0$.

The volume of $g_0$ is infinite, and the sectional curvature of $g_0$ is bounded.

Let $\pi\colon M_1\to M_0$ be a Riemannian covering and $V_1$ be a connected component of $\pi^{-1}(V_0)$.
Then $\pi_1\colon V_1\to V_0$ is a Riemannian covering,
and it is amenable since the fundamental group of $V_0$ is amenable.
Therefore $\lambda_0(V_1)=\lambda_0(V_0)$, by \cite[Theorem 1.2]{BMP}.
Since $\lambda_0(M_1)\le\lambda_0(V_1)=0$,
we conclude that $\lambda_0(M_1)=0$.
It follows that $\lambda_0(M_1)=\lambda_0(M_0)=0$, regardless of whether $\pi$ is amenable or not.

The example is very much in the spirit of the surface $S_\alpha$ (for $0<\alpha<1$), discussed on page 104 of \cite{Br2}.
Note that $S_\alpha$ is complete with finite area and bounded curvature.
\end{exa}

We see in \cref{gefin}.\ref{gefib} and \cref{coco}.\ref{cocob}
that the essential spectrum can be in the way of the bottom of the spectrum to grow.
One aspect of this is revealed in the first of the following two observations.

\begin{prop}\label{equess}
In our setup of Riemannian coverings,
\begin{enumerate}
\item\label{equa}
if $\pi$ is infinite and $\lambda_0(M_1)=\lambda_0(M_0)$, then $\lambda_0(M_1)=\lambda_{\ess}(M_1)$.
\item\label{equb}
if $\lambda_0(M_0)=\lambda_0(M)$, then $\lambda_0(M_1)=\lambda_0(M_0)$.
\end{enumerate}
\end{prop}

The case in \cref{equess}.\ref{equa}, where the deck transformation group of $\pi$ is infinite, is also a consequence of \cite[Corollary 1.3]{Po}.
The proof of \cref{equess}.\ref{equb} is trivial:
By applying \eqref{geq} to $\pi$ and $\pi_1$,
we see that $\lambda_0(M_1)$ is pinched between $\lambda_0(M_0)$ and $\lambda_0(M)$.

The lower bound on the Ricci curvature, required in \cref{notam}, is used in two instances.
First, we need that positive eigenfunctions of the Laplacian satisfy a Harnack inequality.
To that end, we employ the Harnack inequality of Cheng and Yau (see \eqref{cy}).
Second, in the proof of \cref{modbus}, we use Buser's \cref{bulem} below.
Both, the Harnack inequality of Cheng and Yau and Buser's lemma,
require a lower bound on the Ricci curvature.
However, as we already mentioned further up,
we do not know whether \cref{notam} would hold without assuming it.

\begin{que}
Are there non-amenable Riemannian coverings $\pi\colon M_1\to M_0$ of complete and connected Riemannian manifolds $M_0$ and $M_1$, such that $\lambda_{\ess}(M_0)>\lambda_0(M_0)$ and $\lambda_0(M_1)=\lambda_0(M_0)$.
\end{que}

{\bf Structure of the article.}
In \cref{secpre}, we collect some preliminaries about Schr\"odinger operators and the geometry of Riemannian manifolds.
The volume estimate in \cref{secbus} is the basis of our discussion of the amenability of coverings.
Much of the argumentation in this section follows Buser's \cite[Section 4]{Bu}.
In \cref{secfin}, we prove a generalized version of \cref{notam} for Schr\"odinger operators,
where the potential $V$ and its derivative $dV$ are assumed to be bounded.
Furthermore, \cref{secfin} contains the outstanding proofs of \eqref{exafire}, \cref{coco}.\ref{cocob}, and \cref{equess}.\ref{equa}.

\section{Preliminaries}
\label{secpre}

Let $M$ be a Riemannian manifold of dimension $m$
and $V\colon M\to\R$ be a smooth potential.
We denote by $\Delta$ the Laplace operator of $M$
and by $S=\Delta+V$ the Schr\"odinger operator associated to $V$.
We say that a smooth function $\vf$ on $M$ (not necessarily square integrable)
is a \emph{$\lambda$-eigenfunction} if it solves $S\vf=\lambda\vf$.

For a point $x\in M$, subset $A\subseteq M$, and radius $r>0$,
we denote by $B(p,r)$ the open geodesic ball of radius $r$ around $x$ and by
\begin{align}
  A^r = \{p\in M\mid d(p,A)<r\}
\end{align}
the open neighborhood of radius $r$ around $A$, respectively.

For a Lipschitz function $f$ on $M$ with compact support, we call
\begin{align}\label{raylei}
  R(f) = \frac{\int_M\|\grad f\|^2+Vf^2}{\int_Mf^2}
\end{align}
the \emph{Rayleigh quotient} of $f$ and
\begin{align}\label{bottom}
  \lambda_0(M,V) = \inf R(f)
\end{align}
the \emph{bottom of the spectrum of $(M,V)$}.
Here the infimum is taken over all non-vanishing Lipschitz functions on $M$ with compact support.
In the case of the Laplacian, that is, $V=0$, we write $\lambda_0(M)$ instead of $\lambda_0(M,0)$ and call $\lambda_0(M)$ the \emph{bottom of the spectrum of $M$}.
If $M$ is complete and $V$ is bounded from below,
then $\lambda_0(M,V)$ is the minimum of the spectrum of $S$, more precisely,
of the closure of $S$ on $C^\infty_c(M)$ in $L^2(M)$.
We call
\begin{equation}\label{bottome}
  \lambda_{\ess}(M,V) = \sup_K\lambda_0(M\setminus K,V),
\end{equation}
where the supremum is taken over all compact subsets $K$ of $M$,
the \emph{bottom of the essential spectrum of $(M,V)$}.
In the case of the Laplacian, that is, $V=0$, we write $\lambda_{\ess}(M)$ instead of $\lambda_{\ess}(M,0)$ and call $\lambda_{\ess}(M)$ the \emph{bottom of the essential spectrum of $M$}.
If $M$ is complete and $V$ is bounded from below,
then $\lambda_{\ess}(M,V)$ is the minimum of the essential spectrum of $S$.

For a Borel subset $A\subseteq M$, we denote by $|A|$ the volume of $A$.
Similarly, for a submanifold $N$ of $M$ of dimension $n<m$, we let $|N|$ be the $n$-dimensional Riemannian volume of $N$.
We call
\begin{align}\label{checon}
  h(M) = \inf\frac{|\partial A|}{|A|}
  \hspace{3mm}\text{and}\hspace{3mm}
  h_{\ess}(M) = \sup_K h(M\setminus K)
\end{align}
the \emph{Cheeger constant} and \emph{asymptotic Cheeger constant of $M$}, respectively.
Here the infimum is taken over all compact domains $A\subseteq M$ with smooth boundary $\partial A$
and the supremum over all compact subsets $K$ of $M$.
The respective \emph{Cheeger inequality} asserts that
\begin{align}\label{chein}
  \lambda_0(M) \ge \frac14 h^2(M)
  \hspace{3mm}\text{and}\hspace{3mm}
  \lambda_{\ess}(M) \ge \frac14 h_{\ess}^2(M).
\end{align}
The \emph{Buser inequality} is a converse to Cheeger's inequality.
In the case where $M$ is non-compact, complete, and connected with $\Ric_M\ge(1-m)b^2$, where $b\ge0$,
it asserts that 
\begin{align}\label{ichein}
  \lambda_0(M) \le C_{1,m} b h(M).
\end{align}
See \cite[Theorem 7.1]{Bu}.
Here and below, indices attached to constants indicate the dependence of the constants on parameters.
Thus $C_{1,m}$ indicates that the constant depends on $m$ and that a constant $C_{2,m}$ is to be expected.

For a bounded domain $D\subseteq M$ with smooth boundary, we call
\begin{align}\label{neuche}
  h^N(D) = \inf_A\frac{|\partial A\cap\inte D|}{|A|}
\end{align}
the \emph{Cheeger constant of $D$ with respect to the Neumann boundary condition}.
Here $\inte D$ denotes the interior of $D$, and the infimum is taken over all domains $A\subseteq D$ with smooth intersection $\partial A\cap\inte D$ such that $|A|\le|D|/2$.

\subsection{Renormalizing the Schr\"odinger operator}
\label{suseren}
The idea of renormalizing the Laplacian occurs in \cite[Section 8]{Su} and \cite[Section 2]{Br2}.
The idea also works for Schr\"odinger operators, as explained in \cite[Section 7]{Po}.
More details about what we discuss here can be found in the latter article.

Let $M$ be a Riemannian manifold and $V\colon\R\to M$ be a smooth potential.
Let $\vf$ be a positive $\lambda$-eigenfunction of $S=\Delta+V$ on $M$.
For a Borel subset $A\subseteq M$, we denote by $|A|_\vf$ the $\vf$-volume of $A$,
\begin{align}\label{phivol}
  |A|_\vf = \int_A \vf^2.
\end{align}
Similarly, for a submanifold $N$ of $M$ of dimension $n<m$, we let $|N|_\vf$ be the $n$-dimensional $\vf$-volume of $N$.

We renormalize the Schr\"odinger operator $S=\Delta+V$ of $M$ and consider
\begin{align}
  S_\vf = m_{1/\vf}(S-\lambda)m_\vf 
\end{align}
instead, where $m_\vf$ and $m_{1/\vf}$ denote multiplication by $\vf$ and $1/\vf$ respectively.
Now $S$ with domain $C^\infty_c(M)$ is formally and essentially self-adjoint in $L^2(M,dx)$,
where $dx$ denotes the Riemannian volume element of $M$,
and $S_\vf$ is obtained from $S-\lambda$ by conjugation with $m_{1/\vf}$.
Hence $S_\vf$ with domain $C^\infty_c(M)$ is formally and essentially self-adjoint in $L^2(M,\vf^2dx)$.
By \cite[Proposition 7.1]{Po}, we have
\begin{align}\label{bottom2}
  \lambda_0(M,V) - \lambda = \inf\frac{\int_M\|\grad f\|^2\vf^2}{\int_Mf^2\vf^2},
\end{align}
where the infimum is taken over all non-vanishing smooth functions on $M$ with compact support.
By approximation, it follows easily that we obtain the same infimum by considering non-vanishing Lipschitz functions on $M$ with compact support.

For a bounded domain $A\subseteq M$ with smooth boundary $\partial A$, we set
\begin{align}\label{isop}
  h_\vf(M,A) = \frac{|\partial A|_\vf}{|A|_\vf}.
\end{align}
and call
\begin{align}\label{modche}
  h_\vf(M) = \inf_A h_\vf(M,A),
  \hspace{3mm}\text{and}\hspace{3mm}
  h_{\vf,\ess}(M) = \sup_K h_\vf(M\setminus K)
\end{align}
the \emph{modified Cheeger constant} and \emph{modified asymptotic Cheeger constant of $M$}, respectively.
Here the infimum is taken over all compact domains $A\subseteq M$ with smooth boundary $\partial A$
and the supremum over all compact subsets of $M$.
The Cheeger constants in \eqref{checon} correspond to the case $\vf=1$.
By \cite[Corollaries 7.2 and 7.3]{Po}, we have the \emph{modified Cheeger inequalities}
\begin{align}\label{modche2}
  \lambda_0(M,V) - \lambda \ge h_\vf(M)^2/4
  \hspace{3mm}\text{and}\hspace{3mm}
  \lambda_{\ess}(M,V) - \lambda \ge h_{\vf,\ess}(M)^2/4.
\end{align}
In particular, if $\lambda=\lambda_0(M,V)$, then $h_\vf(M)=0$.

\subsection{Volume comparison}
\label{susevol}
Let $H^m$ be the hyperbolic space of dimension $m$ and sectional curvature $-1$,
and denote by $\beta_m(r)$ the volume of geodesic balls of radius $r$ in $H^m$.

\begin{thm}[Bishop-Gromov inequality]\label{grobisi}
Let $M$ be a complete Riemannian manifold of dimension $m$ and $\Ric_M\ge1-m$,
and let $x$ be a point in $M$.
Then
\begin{align*}
  \frac{|B(x,R)|}{|B(x,r)|} \le \frac{\beta_m(R)}{\beta_m(r)}
\end{align*}
for all $0<r<R$.
In particular, $|B(x,r)|\le \beta_m(r)$ for all $r>0$.
\end{thm}

We say that a subset $D\subseteq M$ is \emph{star-shaped with respect to $x\in D$} if, for any $z\in D$ and minimal geodesic $\gamma\colon[0,1]\to M$ from $x$ to $z$, we have $\gamma(t)\in D$ for all $0\le t\le 1$.
Observing that Buser's proof of Lemma 5.1 in \cite{Bu} does not use the compactness of the ambient manifold $M$,
but only the lower bound for its Ricci curvature, his arguments yield the following estimate.

\begin{lem}[Buser]\label{bulem}
Let $M$ be a complete Riemannian manifold of dimension $m$ and $\Ric_M\ge1-m$.
Let $D\subseteq M$ be a domain which is star-shaped with respect to $x\in D$.
Suppose that $B(x,r)\subseteq D\subseteq B(x,2r)$ for some $r>0$.
Then
\begin{align*}
  h^{N}(D) \ge C_{m,r} = \frac1r C_{2,m}^{1+r},
\end{align*} 
where $0<C_{2,m}<1$.
\end{lem}

\subsection{Separated sets}
\label{susesep}
Given $r>0$,
we say that a subset $X\subseteq M$ is \emph{$r$-separated} if $d(x,y)\ge r$ for all points $x\ne y$ in $X$.
An $r$-separated subset $X\subseteq M$ is said to be \emph{complete} if $\cup_{x\in X}B(x,r)=M$. 
Any $r$-separated subset $X\subseteq M$ is contained in a complete one.

We assume now again that $M$ is complete of dimension $m$ with  $\Ric_M\ge1-m$.
For $r>0$ given, we let $X\subseteq M$ be a complete $2r$-separated subset.
For $x\in X$, we call
\begin{align}\label{dirdom}
  D_x = \{ z\in M \mid \text{$d(z,x) \le d(z,y)$ for all $y\in X$}\}
\end{align}
the \emph{Dirichlet domain} about $x$.
Since $X$ is complete as a $2r$-separated subset of $M$,
\begin{align}\label{dirdom2}
  B(x,r) \subseteq D_x \subseteq B(x,2r)
\end{align}
for all $x\in X$.
We therefore get from \cref{grobisi} that
\begin{align}\label{grobis}
  |D_x| \le |B(x,2r)| \le \frac{\beta_m(2r)}{\beta_m(r)}|B(x,r)|.
\end{align}
Furthermore, for any $x\in X$, $z\in D_x$, and minimal geodesic $\gamma\colon[0,1]\to M$ from $x$ to $z$,
we have the strict inequality $d(\gamma(t),x)<d(\gamma(t),y)$ for all $0\le t<1$ and $y\in X$ different from $x$.
In particular, $D_x$ is star-shaped.
Using \cref{bulem}, we conclude that 
\begin{align}\label{hndx}
  h^N(D_x)\ge C_{m,r}\hspace{2mm} \text{for all $x\in X$.}
\end{align}

\subsection{Distance functions}
\label{susedis}
Suppose that $M$ is complete and connected.
Let $K\subseteq M$ be a closed subset and $r>0$.
Define a function $f=f_{K,r}$ on $M$ by
\begin{align*}
  f(x) =
  \begin{cases}
  d(x,K) \vspace{1mm}&\text{if $d(x,K)\le r$,} \\
  r & \text{if $d(x,K)\ge r$.}
  \end{cases}
\end{align*} 
Then $f$ is a Lipschitz function with Lipschitz constant $1$.
A theorem of Rademacher says that the set $\mathcal R$ of points $x\in M$,
such that $f$ is differentiable at $x$, has full measure in $M$.
Clearly, $\|\grad f(x)\|\le1$ for all $x\in\mathcal R$. 

\begin{lem}\label{regdis}
If $x$ is a point in $\mathcal R$ such that $\grad f(x)\ne0$, then $x$ belongs to $K^r\setminus K$, $\grad f(x)$ has norm one, and there is a unique minimizing geodesic from $x$ to $K$.
Moreover, $\partial K^r$ is disjoint from $\mathcal R$.
\end{lem}

\begin{proof}
Let $c$ be a smooth curve through $x$ such that $c'(0)=\grad f(x)$.
Then $(f\circ c)(t)<f(x)$ for all $t<0$ sufficiently close to $0$ and $(f\circ c)(t)>f(x)$ for all $t>0$ sufficiently close to $0$.
Hence $x\notin K$ since $f\ge0$ and $x\notin M\setminus K^r$ since $f\le r$.
Therefore $x\in K^r\setminus K$, that is, $0<f(x)=d(x,K)<r$.
Let $\gamma\colon[0,f(x)]\to M$ be a minimizing unit speed geodesic from $x$ to $K$.
Then $(f\circ\gamma)(t)=f(x)-t$ for all $0\le t\le f(x)$, hence
\begin{align*}
  \la\grad f(x),\gamma'(0)\ra = (f\circ\gamma)'(0) = -1. 
\end{align*}
Since $\|\grad f(x)\|\le1$ and $\|\gamma'(0)\|=1$, we get that $\grad f(x)=-\gamma'(0)$ and hence that $\gamma$ is unique
and that $\|\grad f(x)\|=1$.

For $x\in\partial K^r\cap\mathcal R$ and $\gamma\colon[0,f(x)]\to M$ a minimizing unit speed geodesic from $x$ to $K$,
we would have $-1=(f\circ\gamma)'(0)=\la\grad f(x),\gamma'(0)\ra$, hence that $\grad f(x)\ne0$,
contradicting the first part of the lemma.
\end{proof}

By the same reason as in the last part of the above proof, we get that a point on the boundary of $K$,
which is the endpoint of a minimizing geodesic from some point $x\in M\setminus K$ to $K$,
does not belong to $\mathcal R$.

\subsection{Harnack inequalities}
\label{susehar}
We say that a positive function $\vf$ on $M$ satisfies a \emph{Harnack estimate} if there is a constant $C_{\vf}\ge1$ such that
\begin{align}\label{harnack}
  \sup_{B(x,r)}\vf^2 \le C_\vf^{r+1} \inf_{B(x,r)}\vf^2
\end{align}
for all $x\in M$ and $r>0$.

Suppose now that $M$ is complete with $\Ric_M\ge(1-m)b^2$,
that $|V|$ and $\|\nabla V\|$ are bounded, and that $\vf$ is a positive $\lambda$-eigenfunction of $S=\Delta+V$ on $M$. 
By the estimate of Cheng and Yau \cite[Theorem 6]{CY}, we then have
\begin{align}\label{cy}
  \frac{\|\nabla\vf(x)\|}{\vf(x)}
  &\le C_{3,m}\max\{\|V-\lambda\|_\infty/b, \|\nabla V\|_\infty^{1/3},b\}
\end{align}
for all $x\in M$ (with $m_1=m_4=c=0$, $m_2=m_5=\|V-\lambda\|_\infty$, $m_3=\|\nabla V\|_\infty$, and $a=\infty$ in loc.\,cit.).
In particular, $\vf$ satisfies a Harnack estimate \eqref{harnack}.
Notice that $\Delta$ and $\lambda$ rescale by $1/s$ if the Riemannian metric of $M$ is scaled by $s>0$.
To keep $\vf$ as an eigenfunction, $V$ must therefore also be rescaled by $1/s$.

\section{Modified Buser inequality}
\label{secbus}

Following Buser's arguments in \cite[Section 4]{Bu}, we prove the following estimate.

\begin{lem}\label{modbus}
Let $M$ be a complete and connected Riemannian manifold with Ricci curvature bounded from below
and $\vf>0$ be a smooth function on $M$ which satisfies a Harnack inequality.
Suppose that $h_\vf(M)=0$, and let $\ve,r>0$ be given.
Then there exists a bounded open subset $A\subseteq M$ such that
\begin{align*}
  |A^r\setminus A|_\vf < \ve |A|_\vf.
\end{align*}
\end{lem}

\begin{proof}
Renormalizing the metric of $M$ if necessary, we assume throughout the proof that $\Ric_M\ge1-m$ and let $\beta=\beta_{m}$ (see \cref{susevol}), where $m=\dim M$.

Let $\ve,r>0$ be given.
Recall the constants $C_{m,r}$ and $C_\vf$ from \cref{bulem} and \eqref{harnack}.
Let $A\subseteq M$ be a (non-empty) bounded domain with smooth boundary such that
\begin{align}\label{a}
  \frac{2\beta(4r)C_\vf^{6r+3}}{\beta(r)C_{m,r}} h_\vf(M,A) < \ve,
\end{align}
where $h_\vf(M,A)$ is the isoperimetric ratio of $A$ as in \eqref{isop}.
We partition $M$ into the sets
\begin{align}
  A_+ &= \{ x\in M \mid |A\cap B(x,r)|_\vf > \frac1{2C_\vf^{r+1}}|B(x,r)|_\vf \}, \label{ax1} \\
  M_0 &= \{ x\in M \mid |A\cap B(x,r)|_\vf = \frac1{2C_\vf^{r+1}}|B(x,r)|_\vf \}, \label{ax2} \\
  M_- &= \{ x\in M \mid |A\cap B(x,r)|_\vf < \frac1{2C_\vf^{r+1}}|B(x,r)|_\vf \}. \label{ax3} 
\end{align}
Clearly, $|A\cap D_x|\ne0$ for all $x\in A_+\cup M_0$.
Since $|B(x,r)|_\vf$ and $|A\cap B(x,r)|_\vf$ depend continuously on $x$,
a path from $M_-$ to $A_+$ will pass through $M_0$.
Since $A$ is bounded, $A_+$ and $M_0$ are bounded.
Moreover, $\partial A_+\subseteq M_0$, $A_+$ and $M_-$ are open, and $M_0$ is closed, hence compact.
We will show that $A_+$ satisfies an inequality as required in \cref{modbus}.
By passing from $A$ to $A_+$, we get rid of a possibly \lq\lq hairy structure\rq\rq{} along the \lq\lq outer part\rq\rq{} of $A$.
We pay by possibly loosing regularity of the boundary.

We now choose a $2r$-separated subset $X$ of $M$ as follows.
We start with a $2r$-separated subset $X_0\subseteq M_0$ such that $M_0$ is contained in the union of the balls $B(x,2r)$ with $x\in X_0$.
(If $M_0=\emptyset$, then $X_0=\emptyset$.)
We extend $X_0$ to a $2r$-separated subset $X_0\cup X_+$ of $M_0\cup A_+$ such that $M_0\cup A_+$ is contained in the union of the balls $B(x,2r)$ with $x\in X_0\cup X_+$.
(If $A_+=\emptyset$, then $X_+=\emptyset$.)
We finally extend $X_0\cup X_+$ to a complete $2r$-separated subset $X=X_0\cup X_+\cup X_-$ of $M$.
(If $M_-=\emptyset$, then $X_-=\emptyset$.)
By definition, $X_+\subseteq A_+$ and $X_-\subseteq M_-$.
Since $A$ is bounded and $|A\cap B(x,r)|\ne0$ for all $x\in X_0\cup X_+$, the sets $X_0$ and $X_+$ are finite.
By the same reason, the set $Y$ of $x\in X_-$ with $|A\cap B(x,r)|_\vf\ne0$ is finite.

The neighborhood $M_0^{2r}$ is covered by the balls $B(x,4r)$ with $x\in X_0$.
Using \cref{grobisi}, \eqref{harnack}, and \eqref{ax2}, we therefore get
\begin{align*}
  |M_0^{2r}|_\vf
  &\le \sum_{x\in X_0} |B(x,4r)|_\vf \\
  &\le \frac{\beta(4r)C_\vf^{4r+1}}{\beta(r)} \sum_{x\in X_0} |B(x,r)|_\vf \\
  &=  \frac{2\beta(4r)C_\vf^{5r+2}}{\beta(r)} \sum_{x\in X_0} |A\cap B(x,r)|_\vf.
\end{align*}
For $x\in X_0\subseteq M_0$, we have $|A\cap B(x,r)|\le|B(x,r)|/2$
and hence
\begin{align*}
  \frac{|\partial A\cap B(x,r)|}{|A\cap B(x,r)|} \ge h^N(B(x,r))
\end{align*}
with $h^N(B(x,r))$ according to \eqref{neuche}.
Applying \cref{bulem} to $D=B(x,r)$, we therefore obtain
\begin{align*}
  \frac{|\partial A\cap B(x,r)|_\vf}{|A\cap B(x,r)|_\vf}\ge \frac1{C_\vf^{r+1}}\frac{|\partial A\cap B(x,r)|}{|A\cap B(x,r)|}
  \ge \frac{C_{m,r}}{C_\vf^{r+1}}.
\end{align*}
Hence
\begin{equation}\label{ax5}
\begin{split}
  |M_0^{2r}|_\vf 
  &\le \frac{2\beta(4r)C_\vf^{6r+3}}{\beta(r)C_{m,r}} \sum_{x\in X_0} |\partial A\cap B(x,r)|_\vf \\
  &\le \frac{2\beta(4r)C_\vf^{6r+3}}{\beta(r)C_{m,r}} |\partial A|_\vf \\
  &= \frac{2\beta(4r)C_\vf^{6r+3}}{\beta(r)C_{m,r}} h_\vf(M,A) |A|_\vf
  \le \ve |A|_\vf,
\end{split}
\end{equation}
where we use that $h_\vf(M,A)$ satisfies \eqref{a}.

Since any curve from $A_+$ to $M_-$ passes through $M_0$, $A_+$ has distance at least $2r$ to $M_-\setminus M_0^{2r}$.
Hence $M_-\setminus M_0^{2r}$ is covered by the Dirichlet domains $D_x$ with $x\in X_-$. 

With $Y$ as above, we let $Z=X_0\cup Y$.
Using \eqref{ax2} and \eqref{ax3}, we have
\begin{align*}
  \frac{|A\cap B(x,r)|}{|B(x,r)|}
  \le C_\vf^{r+1} \frac{|A\cap B(x,r)|_\vf}{|B(x,r)|_\vf} \le \frac12
\end{align*}
for any $x\in Z$.
Letting $A^c=M\setminus A$, we obtain
\begin{equation*}
\begin{split}
  |A^c\cap D_x|
  &\ge |A^c\cap B(x,r)| \ge \frac12 |B(x,r)| \\
  &\ge \frac{\beta(r)}{2\beta(2r)}|D_x| \ge  \frac{\beta(r)}{2\beta(2r)}|A\cap D_x| > 0.
\end{split}
\end{equation*}
for any $x\in Z$, where we use in the third inequality that $D_x\subseteq B(x,2r)$.
With the constant $C_{m,r}$ as in \cref{bulem}, we therefore get
\begin{equation}\label{az4}
\begin{split}
  C_{m,r}
  &\le h^N(D_x) \\
  &\le \frac{|\partial A\cap\inte D_x|}{\min\{|A\cap D_x|,|A^c\cap D_x|\}} \\
  &\le \frac{2\beta(2r)}{\beta(r)} \frac{|\partial A\cap\inte D_x|}{|A\cap D_x|} \\
  &\le \frac{2\beta(2r)C_\vf^{2r+1}}{\beta(r)} \frac{|\partial A\cap\inte D_x|_\vf}{|A\cap D_x|_\vf}
\end{split}
\end{equation}
for any $x\in Z$, where we use again, now in the last inequality, that $D_x\subseteq B(x,2r)$.
Using \eqref{az4} and \eqref{a}, we conclude that
\begin{equation}\label{az5}
\begin{split}
  |A\cap(M_-\setminus M_0^{2r})|_\vf
  &\le \sum_{x\in Z}|A\cap D_x|_\vf \\
  &\le \frac{2\beta(2r)C_\vf^{2r+1}}{\beta(r)C_{m,r}} \sum_{x\in Z}|\partial A\cap\inte D_x|_\vf \\
  &\le \frac{2\beta(2r)C_\vf^{2r+1}}{\beta(r)C_{m,r}} |\partial A|_\vf \\
  &= \frac{2\beta(2r)C_\vf^{2r+1}}{\beta(r)C_{m,r}} h_\vf(M,A) |A|_\vf
  \le \ve |A|_\vf,
\end{split}
\end{equation}
where we use \eqref{a} in the last step, recalling that $C_\vf\ge1$.

Since $A\subseteq A_+\cup M_0^{2r}\cup(A\cap(M_-\setminus M_0^{2r}))$, we obtain
\begin{align*}
  |A_+|_\vf
  &\ge |A|_\vf - |M_0^{2r}|_\vf - |A\cap(M_-\setminus M_0^{2r})|_\vf \\
  &\ge (1-2\ve)|A|_\vf.
\end{align*}
In particular, $A_+$ is not empty.
Since $A_+^{2r}\setminus A_+\subseteq M_0^{2r}$, we conclude that
\begin{align*}
  |A_+^{2r}\setminus A_+|_\vf \le |M_0^{2r}|_\vf \le \ve|A|_\vf
  \le \frac{\ve}{1-2\ve}|A_+|_\vf.
\end{align*}
In conclusion, $A_+$ is a bounded open subset of $M$ that satisfies an inequality as asserted in \cref{modbus}, albeit with $2r$ and $2\ve$ in place of $r$ and $\ve$ (assuming w.l.o.g.\;that $\ve<1/4$). 
\end{proof}

Whereas $\ve>0$ should be viewed as small, the number $r$ is large in our application of \cref{modbus} (see \eqref{rlarge}).
The difference to Buser's discussion lies in the fact that in \cref{modbus}, for $\ve$ and $r$ are given, the domain $A$ is chosen according to \eqref{a}.

\begin{rem}\label{modbus1}
Let $M$ be a non-compact, complete, and connected Riemannian manifold of dimension $m$ with $\Ric_M\ge(1-m)b^2$.
Let $V\colon M\to\R$ be a smooth potential on $M$, and assume that $V$ and $\nabla V$ are bounded.
Let $\vf$ be a positive $\lambda$-eigenfunction of the associated Schr\"odinger operator $S$ on $M$.
Following the above line of proof and Buser's arguments at the end of his short proof of Theorem 1.2 in \cite{Bu}, 
one obtains inequalities of the form
\begin{equation}\label{modbus1a}
\begin{split}
  \lambda_0(M,V) - \lambda
  &\le  C'_{m,\|V-\lambda\|_\infty,\|\nabla V\|_\infty}\max\{bh_\vf(M),h_\vf(M)^2\}, \\
  \lambda_{\ess}(M,V) - \lambda
  &\le  C'_{m,\|V-\lambda\|_\infty,\|\nabla V\|_\infty}\max\{bh_{\vf,\ess}(M),h_{\vf,\ess}(M)^2\}.
\end{split}
\end{equation}
To get rid of the squares $h_\vf(M)^2$ and $h_{\vf,\ess}(M)^2$, respectively,
we change Buser's argument at the end of his proof of \cite[Theorem 7.1]{Bu} and estimate
\begin{equation*}
\begin{split}
  h_\vf(M), h_{\vf,\ess}(M) &\le \sup_{x\in M}h_\vf(B(x,1)) \\
  &\le C_\vf^2 \sup_{x\in M}h(B(x,1)) \\
  &\le 2 C_\vf^2 \sup_{x\in M}\lambda_0(B(x,1))^{1/2} \\
  &\le 2 C_\vf^2 \lambda_0(B)^{1/2}
  \le bC''_{m,\|V-\lambda\|_\infty,\|\nabla V\|_\infty}, 
\end{split}
\end{equation*}
where we use the definition of $h_\vf$ and $h_{\vf,\ess}$ as in \eqref{modche},
the Harnack constant of $\vf$ as in \eqref{harnack},
the Cheeger inequality \eqref{chein},
and Cheng's \cite[Theorem 1.1]{Ch},
where $B$ denotes a ball of radius $1$ in the $m$-dimensional hyperbolic space of sectional curvature $-b^2$.
We finally arrive at the inequalities
\begin{equation}\label{genbus}
\begin{split}
  \lambda_0(M,V) - \lambda
  &\le C_{m,\|V-\lambda\|_\infty,\|\nabla V\|_\infty} b h_\vf(M), \\
  \lambda_{\ess}(M,V) - \lambda
  &\le C_{m,\|V-\lambda\|_\infty,\|\nabla V\|_\infty} b h_{\vf,\ess}(M),
\end{split}
\end{equation}
which extend Buser's \cite[Theorem 7.1]{Bu}.
The dependence of $C_{m,\|V-\lambda\|_\infty,\|\nabla V\|_\infty}$ on $C_{3,m}$ (as in \eqref{cy}), $\|V-\lambda\|_\infty$, and $\|\nabla V\|_\infty$ is exponential in our approach and, in particular, exponential in $\lambda$.
Therefore the use of the estimates seems to be restricted.
However, together with \eqref{modche2}, they have at least the consequence that
$\lambda_0(M,V)=\lambda$ if and only if $h_\vf(M)=0$
and that $\lambda_{\ess}(M,V)=\lambda$ if and only if $h_{\vf,\ess}(M)=0$.
\end{rem}

\section{Back to Riemannian coverings}
\label{secfin}

We return to the situation of a Riemannian covering as in the introduction.
Suppose that the Ricci curvature of $M_0$ is bounded from below.
Let $V_0$ be a smooth potential on $M_0$ with $\|V_0\|_\infty,\|\nabla V_0\|_\infty<\infty$
and set $V_1=V_0\circ\pi$.
Let $\lambda=\lambda_0(M_0,V_0)$ and $\vf_0$ be a positive $\lambda$-eigenfunction of $S_0=\Delta+V_0$ on $M_0$.
Then $\vf=\vf_0\circ\pi$ is a positive $\lambda$-eigenfunction of $S_1=\Delta+V_1$ on $M_1$.

\begin{thm}\label{notams}
If $\lambda_{\ess}(M_0,V_0)>\lambda_0(M_0,V_0)$,
then $\lambda_0(M_1,V_1)=\lambda_0(M_0,V_0)$ if and only if the covering $\pi\colon M_1\to M_0$ is amenable.
\end{thm}

Consider the following three implications:
\begin{enumerate}
\item
If $\pi\colon M_1\to M_0$ is amenable, then $\lambda_0(M_1,V_1)=\lambda_0(M_0,V_0)$.
\item
If $\lambda_0(M_1,V_1)=\lambda_0(M_0,V_0)$, then $h_\vf(M_1)=0$.
\item
If $h_\vf(M_1)=0$, then $\pi\colon M_1\to M_0$ is amenable.
\end{enumerate}
The first one is \cite[Theorem 1.2]{BMP}
and the second is an immediate consequence of \eqref{modche2}.
These two assertions hold without any assumptions on the curvature of $M$ and the potential $V$.
The third one does not hold without any further assumptions.
We require that the Ricci curvature of $M_0$ is bounded from below,
that the potential $V_0$ and its derivative $dV_0$ are bounded,
and that $\lambda_{\ess}(M_0,V_0)>\lambda_0(M_0,V_0)$.
To prove \cref{notams}, and therewith also \cref{notam},
it remains to establish the third implication under these additional assumptions.
We need to prove that the right action of $\Gamma_0$ on $\Gamma_1\backslash\Gamma_0$ is amenable. 
To that end, we will show that the F\o{}lner criterion for amenability is satisfied.

\begin{fc}\label{folner}
The right action of a countable group $\Gamma$ on a countable set $X$ is amenable if and only if, for any finite subset $G\subseteq\Gamma$ and $\ve>0$, there is a finite subset $F\subseteq X$ such that
\begin{align*}
  \#(F\setminus Fg)<\ve\#(F) \hspace{2mm}\text{for all $g\in G$.}
\end{align*}
\end{fc} 

\begin{proof}[Proof of \cref{notams}]
Since $\lambda_{\ess}(M_0,V_0)>\lambda_0(M_0,V_0)$, there is a compact domain $K\subseteq M_0$ such that
\begin{align}\label{ess0}
  \lambda_0(M_0\setminus K,V_0) > \lambda_0(M_0,V_0).
\end{align}
Since $\pi\colon M_1\setminus\pi^{-1}(K)\to M_0\setminus K$ is a Riemannian covering, we have
\begin{align}\label{ess1}
  \lambda_0(M_1\setminus\pi^{-1}(K),V_1) \ge \lambda_0(M_0\setminus K,V_0).
\end{align}
Note that the manifolds $M_0\setminus K$ and $M_1\setminus\pi^{-1}(K)$ might be not connected, but the assertion still holds since the inequality applies to each component of $M_0\setminus K$ and connected component of $M_1\setminus\pi^{-1}(K)$ over it.

Let $\chi_0$ be a smooth cut-off function on $M_0$ which is equal to $0$ on a neighborhood of $K$ in $M_0$ and equal to $1$ outside a compact domain $K_0\subseteq M_0$ and set $\chi=\chi_0\circ\pi$.

\begin{lem}\label{count}
For all $r,\ve>0$,
there is a bounded open subset $A\subseteq M_1$ and a point $x\in K_0$ such that $\pi^{-1}(x)\cap A\ne\emptyset$ and
\begin{align*}
  \frac{\#\big(\pi^{-1}(x)\cap(A^r\setminus A)\big)}{\#\big(\pi^{-1}(x)\cap A\big)}
  < \ve.
\end{align*}
\end{lem}

\begin{proof}
Since $M_1$ is complete with Ricci curvature bounded from below and $h_\vf(M_1)=0$,
\cref{modbus} implies that there exist bounded open subsets $A_n\subseteq M_1$ such that
\begin{align}
  \frac{|A_n^r\setminus A_n|_\vf}{|A_n|_\vf} < \frac1n.
\end{align}
Let $f_n$ be the Lipschitz function on $M_1$ with compact support defined by
\begin{align}
  f_n(x) = \begin{cases}
  1-d(x,A_n)/r \hspace{1mm}&\text{for $x\in A_n^r$,} \\ 0 &\text{for $x\in M_1\setminus A_n^r$.}
  \end{cases}
\end{align}
For the $\vf$-Rayleigh quotient of $f_n$, we have
\begin{equation}
\begin{split}
  R_\vf(f_n)
  &= \frac{\int_{M_1} \|\grad f_n\|^2\vf^2}{\int_{M_1} f_n^2\vf^2} \\
  &\le \frac{\int_{A_n^r\setminus A_n} \|\grad f_n\|^2\vf^2}{\int_{A_n} f_n^2\vf^2} \\
  &=\frac1{r^2}\frac{|U_r(A_n)\setminus A_n|_\vf}{|A_n|_\vf}
  \le \frac1{nr^2}.
\end{split}
\end{equation}
Normalize $f_n$ to $g_n=f_n/\|f_n\|$, where $\|f_n\|$ denotes the modified $L^2$-norm of $f_n$,
that is, $\|f_n\|^2=\int_{M_1}f_n^2\vf^2$.
Then
\begin{align*}
  R_\vf(g_n) = R_\vf(f_n)\le1/nr^2 \to 0.
\end{align*}
Let $\mathcal R\subseteq M_0$ be the subset of full measure such that all $g_n$ are differentiable at all $y\in \pi^{-1}(\mathcal R)$.
Suppose now that
\begin{align*}
  \sum_{y\in\pi^{-1}(x)} \|\grad g_n(y)\|^2
  \ge \ve \sum_{y\in\pi^{-1}(x)} g_n(y)^2
\end{align*}
for all $n\in\N$ and $x\in K_0\cap\mathcal R$.
Since $\pi$ is a Riemannian covering and $\vf$ is constant along the fibers of $\pi$, we then have
\begin{align*}
  \int_{\pi^{-1}(K_0)} \|\grad g_n\|^2\vf^2 \ge \ve \int_{\pi^{-1}(K_0)} g_n^2\vf^2.
\end{align*}
Since $\|g_n\|=1$ and $R_\vf(g_n)\le1/nr^2 \to 0$, we get that
\begin{align*}
  \int_{\pi^{-1}(K_0)} g_n^2\vf^2 \to 0
  \hspace{2mm}\text{and, as a consequence,}\hspace{2mm}
  \int_{M_1\setminus\pi^{-1}(K_0)} g_n^2\vf^2 \to 1.
\end{align*}
Consider now $h_n=\chi g_n$ with $\chi$ as further up.
Then $h_n$ has compact support in $M_1\setminus\pi^{-1}(K)$.
Furthermore,
\begin{align*}
  \int_{M_1} h_n^2\vf^2 = \int_{\pi^{-1}(K_0)} h_n^2\vf^2 + \int_{M_1\setminus\pi^{-1}(K_0)} g_n^2\vf^2 \to 0+1
\end{align*}
and
\begin{equation*}
\begin{split}
  \int_{M_1}\|\grad h_n\|\vf^2
  &\le 2\int_{\pi^{-1}(K_0)} \big(g_n^2\|\grad\chi\|^2+\chi^2\|\grad g_n\|^2\big)\vf^2 \\
  &\hspace{20mm}+ \int_{M\setminus\pi^{-1}(K_0)} \|\grad g_n\|\vf^2
  \to 0,
\end{split}
\end{equation*}
where we use that $0\le\chi\le1$, that $\grad\chi$ is uniformly bounded, and that $\int_{M_1}\|\grad g_n\|\vf^2\to 0$.
Hence the modified Rayleigh quotients $R_\vf(h_n)\to 0$.
This is in contradiction to \eqref{ess1} since the $h_n$ are Lipschitz functions on $M_1$ with compact support in $M_1\setminus\pi^{-1}(K_0)$.
It follows that there are an $n$ and an $x\in K_0\cap\mathcal R$ such that
\begin{align*}
  \sum_{y\in\pi^{-1}(x)} \|\grad g_n(y)\|^2
  < \ve \sum_{y\in\pi^{-1}(x)} g_n(y)^2.
\end{align*}
Since $g_n=0$ on $M_1\setminus A_n^r$, we must have $\pi^{-1}(x)\cap A_n^r\ne\emptyset$.
Furthermore, since $0\le g_n\le1/\|f_n\|$ and $||\grad g_n||=1/r\|f_n\|$ on $\pi^{-1}(\mathcal R)\cap(A_n^r\setminus A_n)$, we conclude that
\begin{align*}
  \frac1{r^2\|f_n\|^2} \#(\pi^{-1}(x)\cap(A_n^r\setminus A_n))
  \le \frac{\ve}{\|f_n\|^2} \#(\pi^{-1}(x)\cap A_n^r).
\end{align*}
This yields that
\begin{align*}
  \#(\pi^{-1}(x)\cap(A_n^r\setminus A_n)) < \ve r^2 \#(\pi^{-1}(x)\cap A_n^r).
\end{align*}
Since $A_n^r$ is the disjoint union of $A_n$ with $A_n^r\setminus A_n$,
we conclude that
\begin{align*}
  \#(\pi^{-1}(x)\cap(A_n^r\setminus A_n)) < \frac{\ve r^2}{1-\ve r^2}(\pi^{-1}(x)\cap A_n)
\end{align*}
as long as $\ve<1/r^2$.
In particular, $\pi^{-1}(x)\cap A_n\ne\emptyset$ if $\ve<1/r^2$.
\end{proof}

We return to the proof of the amenability of the right action of $\Gamma_0$ on $\Gamma_1\backslash\Gamma_0$.
We will use F\o{}lner's criterion \ref{folner}
and let $G\subseteq \Gamma_0$ be a finite subset and $\ve>0$.
We need to show that there is a non-empty finite subset $F\subseteq\Gamma_1\backslash\Gamma_0$ such that
\begin{align*}
  \#(F\setminus Fg) < \ve\#(F) \hspace{2mm}\text{for all $g\in G$.}
\end{align*}
Write $K_0$ as the union of finitely many compact and connected domains $D_i\subseteq M_0$ which are evenly covered with respect to the universal covering $\pi_0\colon M\to M_0$ of $M_0$.
For each $i$, let $B_i$ be a lift of $D_i$ to a leaf of $\pi_0$ over $D_i$.
Then each $B_i$ is a compact subset of $M$ with $\pi_0(B_i)=D_i$.
Since there are only finitely many $B_i$ and all of them are compact, there is a number $r>0$ such that
\begin{align}\label{rlarge}
  d(u,g^{-1}u) < r \hspace{2mm}\text{for all $g\in G$ and $u\in\cup_iB_i$.}
\end{align}
Let $R\subseteq\Gamma_0$ be a set of representatives of the right cosets of $\Gamma_1$ in $\Gamma_0$,
that is, of the elements of $\Gamma_1\backslash\Gamma_0$.
Corresponding to $\ve$ and $r$, choose $x\in K_0$ and $A$ as in \cref{count}.
Fix preimages $u\in M$ and $y=\pi_1(u)\in M_1$ of $x$ under $\pi_0$ and $\pi$, respectively,
and write $\pi_0^{-1}(x) = \Gamma_0u$ as the union of $\Gamma_1$-orbits $\Gamma_1gu$. 
Then $\pi^{-1}(x) = \{\pi(gu) \mid g\in R\}$.
Set
\begin{align*}
 F=\{\Gamma_1h \mid \text{$h\in R$ and $\pi_1(hu)\in\pi^{-1}(x)\cap A$}\}.
\end{align*}
Then $\#(F)=\#(\pi^{-1}(x)\cap A)\ne0$.

Let now $g\in G$ and $h\in R$ with $\Gamma_1h\in F\setminus Fg$.
Then
\begin{align*}
  \pi_1(hu)\in\pi^{-1}(x)\cap A
  \hspace{2mm}\text{and}\hspace{2mm}
  \pi_1(hg^{-1}u)\in\pi^{-1}(x)\setminus A.
\end{align*} 
Since
\begin{align*}
  d(\pi_1(hu),\pi_1(hg^{-1}u)) \le d(hu,hg^{-1}u) = d(u,g^{-1}u) < r
\end{align*}
for all $g\in G$, we get that $\pi_1(hg^{-1}u)\in A^r$.
Hence $\pi_1(hg^{-1}u)$ belongs to $A^r\setminus A$ and therefore
\begin{align*}
  \#(F\setminus Fg)
  &\le \#\big(\pi^{-1}(x)\cap A^r\setminus A)\big) \\
  &< \ve \#\big(\pi^{-1}(x)\cap A\big)
  = \ve \#(F).
\end{align*}
Since $G$ and $\ve$ were arbitrary, we conclude from \cref{folner} that the right action of $\Gamma_0$ on $\Gamma_1\backslash\Gamma_0$ is amenable.
\end{proof}

\begin{proof}[Proof of \cref{coco}.\ref{cocob}]
Let $M_0$ be the interior of a compact manifold $N_0$ as in the definition of conformally compact (in the introduction),
and denote by $g_0$, $h_0$, and $\rho_0$ the corresponding Riemannian metrics and defining function $\rho_0$ of $\partial N_0$. 
Let $X=\grad\rho_0/\|\grad\rho_0\|^2$,
where the gradient of $\rho_0$ is taken with respect to $h_0$.
Since $\partial N_0$ is compact, the flow of $X$ leads to a diffeomorphism of a neighborhood of $\partial N_0$ in $N_0$ with $\partial N_0\times[0,y_0)$
with respect to which $\rho(x,y)=y$ for $(x,y)\in\partial N_0\times[0,y_0)$.
Then
\begin{align*}
  g_0(x,y) = \frac1{y^2}h_0(x,y)
\end{align*}
on $\partial N_0\times(0,y_0)=\partial N_0\times[0,y_0)\cap M_0$.
This is reminiscent of the upper half-space model of the hyperbolic space $H^m$.

From standard formulas for conformal metrics it is now easy to see that,
for all $x_0\in\partial N_0$ and $\ve>0$, there exists a neighborhood $U$ of $(x_0,0)\in\partial N_0\times[0,y_0)$
such that the sectional curvature of each tangent plane at each $(x,y)$ in $U\cap M_0$ is in
\begin{align*}
  (-(\partial_{\nu}\rho(x,0))^2-\ve,-(\partial_{\nu}\rho(x,0))^2+\ve),
\end{align*}
where $\nu$ denotes the inner normal of $N_0$ along $\partial N_0$ with respect to $h_0$.
Note that, for any $r>0$, the $g_0$-ball $B((x_0,y),r))$ is contained in $U\cap M_0$ for all sufficiently small $y>0$.
From Cheng's \cite[Theorem 1.1]{Ch}, we conclude that $\lambda_0(M_0)\le a^2(m-1)^2/4$.

Since $M_0$ is homotopy equivalent to $N_0$, there is a covering $\pi_1\colon N_1\to N_0$
which restricts to the covering $M_1\to M_0$ and such that $M_1$ is the interior of the manifold $N_1$,
but where the boundary $\partial N_1$ of $N_1$ need not be compact anymore.
Nevertheless, lifting $g_0$, $h_0$, and $\rho_0$ to Riemannian metrics $g_1$ on $M_1$, $h_1$ on $N_1$,
and defining function $\rho_1=\rho_0\circ\pi_1$ of $\partial N_1$,
the above statement about sectional curvature remains valid for
\begin{align*}
  \partial N_1\times[0,y_0) = \pi^{-1}(\partial N_0\times[0,y_0)).
\end{align*}
In particular, we have $\lambda_0(M_1)\le a^2(m-1)^2/4$.

Now we are ready for the final step of the proof.
By assumption and \eqref{geq},
\begin{align*}
  a^2(m-1)^2/4 = \lambda_0(M_0) \le \lambda_0(M_1) \le a^2(m-1)^2/4.
\end{align*}
Hence $\lambda_0(M_0) = \lambda_0(M_1)$ as asserted.
\end{proof}

\begin{proof}[Proof of \cref{equess}.\ref{equa}]
By definition, $\lambda_{\ess}(M_1)>\lambda_0(M_1)=:\lambda$
would imply that $\lambda$ does not belong to the essential spectrum of $M_1$.
Hence $\lambda$ would be an eigenvalue of $M_1$ with a square integrable positive eigenfunction $\vf$.
On the other hand, the lift $\psi$ of a positive $\lambda$-eigenfunction from $M_0$ to $M_1$ is also a positive $\lambda$-eigenfunction, but definitely not square integrable since $\pi$ is an infinite covering.
Now by Sullivan's \cite[Theorems 2.7 and 2.8]{Su},
the space of positive, but not necessarily square integrable, $\lambda$-eigenfunctions on $M_1$ is of dimension one.
Hence $\psi$ would be a multiple of $\vf$, a contradiction.
\end{proof}

\begin{proof}[Proof of \eqref{exafire}]
By \cite[Theorem 3.1]{Eb},
each end of $M_0$ has a neighborhood of the form $U=\Gamma_\infty\backslash B$,
where $B$ is a horoball in the universal covering space $M$ of $M_0$ and $\Gamma_\infty\subseteq\Gamma_0$ is the stabilizer of the center $\xi$ of $B$ in the sphere of $M$ at infinity.
Furthermore, $\Gamma_\xi$ leaves the Busemann functions associated to $\xi$ invariant.
We let $b$ be the one such that $\{b=0\}$ is the horosphere $\partial B$.
Then the level sets $\{b=-y\}$, $y>0$, are horospheres foliating $B$.
They are perpendicular to the unit speed geodesics $\gamma_z$ starting in $z\in\{b=0\}$ and ending in $\xi$.
Moreover, $b(\gamma_z(y))=-y$ and $\grad b(\gamma_z(y))=-\dot\gamma_z(y)$.
Since Busemann functions are $C^2$ (see \cite[Proposition 3.1]{HI}),
we obtain a $C^2$-diffeomorphism
\begin{align*}
  \{b=0\}\times(0,\infty) \to B, \quad (z,y) \mapsto \gamma_z(y). 
\end{align*}
Since $\Gamma_\xi$ leaves $b$ invariant, we arrive at a $C^2$-diffeomorphism $U\cong N\times(0,\infty)$,
where $N=\Gamma_\xi\backslash\{b=0\}$
and where the curves $\gamma_x=\gamma_x(y)=(x,y)$ are unit speed geodesics perpendicular to the cross sections $\{y={\rm const}\}$.
The latter lift to the horospheres $\{b={\rm const}\}$ in $B$ and, therefore, have second fundamental form $\le-a$ with respect to the unit normal field $Y=\partial/\partial y$.
In particular, their mean curvature is $\le(1-m)a$ with respect to $Y$.
For the divergence of $Y$, we have
\begin{align*}
  \dive Y = \sum\la\nabla_{E_i}Y,E_i\ra
  = - \sum \la Y, \nabla_{E_i}E_i\ra,
\end{align*}
where $(E_i)$ is a local orthonormal frame.
We choose it such that $E_1=Y$.
Then $\nabla_{E_1}E_1=0$, and we see that $\dive Y$ is the mean curvature of the corresponding cross section with respect to the unit normal field $Y$, hence is $\le(1-m)a$.
All this is well known, but we recall it for convenience.

For a compact domain $A$ in $U$ with smooth boundary $\partial A$ and outer unit normal field $\nu$,
we obtain from the above that
\begin{align*}
  |\partial A| \ge -\int_{\partial A}\la Y,\nu\ra = -\int_A\dive Y \ge a(m-1)|A|.
\end{align*}
Hence the Cheeger constant of $U$ is at least $a(m-1)$.
The claim about $\lambda_{\ess}(M_0)$ now follows from the Cheeger inequality \eqref{chein}.
\end{proof}


\newpage

\begin{thebibliography}{12}

\bibitem{BMP}
W.\,Ballmann, H.\,Matthiesen, P.\,Polymerakis,
On the bottom of spectra under coverings.
\emph{Math. Zeitschrift},
doi.org/10.1007/s00209-017-1925-9.

\bibitem{Br1}
R.\,Brooks,
The fundamental group and the spectrum of the Laplacian.
\emph{Comment. Math. Helv.} {\bf 56} (1981), no. 4, 581--598,
MR656213, Zbl 0495.58029.

\bibitem{Br2}
R.\,Brooks,
The bottom of the spectrum of a Riemannian covering.
\emph{J. Reine Angew. Math.} {\bf 357} (1985), 101--114,
MR783536, Zbl 0553.53027.

\bibitem{Bu}
P.\,Buser,
\emph{A note on the isoperimetric constant.}
Ann. Sci. \'Ecole Norm. Sup. (4) 15 (1982), no. 2, 213--230,
MR0683635, Zbl 0501.53030.

\bibitem{Ch}
S.\,Y.\,Cheng,
Eigenvalue comparison theorems and its geometric applications.
\emph{Math. Z.} {\bf 143} (1975), no. 3, 289--297,
MR0378001, Zbl 0329.53035. 

\bibitem{CY}
S.\,Y.\,Cheng, S.\,T.\,Yau,
Differential equations on Riemannian manifolds and their geometric applications.
\emph{Comm. Pure Appl. Math.} {\bf 28} (1975), no. 3, 333--354.
MR0385749, Zbl 0312.53031. 

\bibitem{Eb}
P.\,Eberlein,
Lattices in spaces of nonpositive curvature,
\emph{Ann. of Math.} {\bf 111} (1980), 435--476.

\bibitem{HI}
E.\,Heintze, H.-C.\,Im Hof,
Geometry of horospheres.
\emph{J. Differential Geom}. {\bf 12} (1977), no. 4, 481--491 (1978).
 
\bibitem{LP}
P.\,D.\,Lax, R.\,S.\,Phillips,
The asymptotic distribution of lattice points in Euclidean and non-Euclidean spaces.
\emph{Toeplitz centennial} (Tel Aviv, 1981), pp. 365--375,
Operator Theory: Adv. Appl., 4, Birkh\"auser, Basel-Boston, Mass., 1982,
MR065029, Zbl 0497.52007.

\bibitem{Ma1}
R.\,Mazzeo,
The Hodge cohomology of a conformally compact metric.
\emph{J.\,Differential Geom.} {\bf 28} (1988), no. 2, 309--339,
MR0961517, Zbl 0656.53042. 

\bibitem{Po}
P.\,Polymerakis,
\emph{On the spectrum of differential operators under Riemannian coverings.}
MPI-Preprint 2018, arxiv.org/abs/1803.03223.

\bibitem{RT}
T.\,Roblin, S.\,Tapie,
\emph{Exposants critiques et moyennabilit\'e.}
G\'eom\'etrie ergodique, 61--92, Monogr. \'Enseign. Math., 43, \'Enseignement Math., Geneva, 2013,
MR3220551, Zbl 1312.53060.

\bibitem{Su}
D.\,Sullivan,
Related aspects of positivity in Riemannian geometry.
\emph{J. Differential Geom.} {\bf 25} (1987), no. 3, 327--351,
MR0882827,  Zbl 0615.53029.

\end{thebibliography}
\end{document}